\documentclass[11pt,twoside]{article}
\usepackage{a4}
\usepackage{amssymb,amsmath,amsthm,latexsym}
\usepackage{amsfonts}
\usepackage{amsmath}
\usepackage{amsfonts}
\usepackage{float}
\usepackage{geometry}
\usepackage{graphicx}
\usepackage{placeins}
\usepackage{longtable}
\usepackage[tableposition=below]{caption}
\usepackage{cite}
\usepackage{tikz}
\newtheorem{theorem}{Theorem}[section]

\newtheorem{definition}[theorem]{Definition}

\newtheorem{lemma} [theorem]{Lemma}

\begin{document}
	\title{Bounds on Atom-Bond Connectivity and Zagreb Indices in Trees with a Given Metric Dimension}
	\author{Waqar Ali$^{1}$, Mohamad Nazri Husin$^{1}$, Muhammad Faisal Nadeem$^{2}$, Muqaddas Jabin$^{3}$}
	\date{}
	\maketitle
	\vspace{-7mm}
	\begin{center}
		
		{\it\small 1 Special Interest Group on Modeling and Data Analytics, Faculty of Computer Science and Mathematics, Universiti Malaysia Terengganu, Kuala Nerus 21030, Terengganu}\\
		{\it\small 2 Department of Mathematics, COMSATS University Islamabad Lahore Campus, Lahore 54000 Pakistan.}\\
		{\it\small 2 Department of Mathematics, University of Okara, N-5 Okara, 56300, Pakistan.}\\
		{	 meharwaqaraali@gmail.com\\}
		
		{	Corresponding author nazri.husin@umt.edu.my\\}
				
		{	mfaisalnadeem@ymail.com\\}
		{muqaddasjaabin@gmail.com\\}
		
	\end{center}
	\begin{abstract}
		Let $\mathbb{G} = (\mathcal{V}, \mathcal{E})$ be a simple connected graph, where $\mathcal{V}$ and $\mathcal{E}$ denote the vertex and edge sets, respectively. The first Zagreb index is defined as $\mathcal{M}_{1}(\mathbb{G}) = \sum_{v \in \mathcal{V}} \zeta_{\mathbb{G}}(v)^2$, while the second Zagreb index is given by $\mathcal{M}_{2}(\mathbb{G}) = \sum_{uv \in \mathcal{E}} \zeta_{\mathbb{G}}(u)\, \zeta_{\mathbb{G}}(v)$, where $\zeta_{\mathbb{G}}(v)$ represents the degree of vertex $v$. Another notable degree-based invariant is the atom-bond connectivity (ABC) index, introduced in chemical graph theory, and defined by
		\[
		ABC(\mathbb{G}) = \sum_{uv \in \mathcal{E}} \sqrt{\frac{\zeta_{\mathbb{G}}(u) + \zeta_{\mathbb{G}}(v) - 2}{\zeta_{\mathbb{G}}(u)\, \zeta_{\mathbb{G}}(v)}}.
		\]
		A fundamental graph parameter, the metric dimension, refers to the minimum number of vertices in a resolving set that uniquely distinguishes all other vertices based on distances. In this work, we investigate the influence of metric dimension on the Zagreb and ABC indices within the class of trees. We derive sharp bounds-both upper and lower for $\mathcal{M}_1$ and $\mathcal{M}_2$, and provide an upper bound for the ABC index, all expressed in terms of the tree's order and its metric dimension. Furthermore, we identify the extremal tree structures that attain these bounds. These findings underscore the role of metric dimension in shaping topological descriptors and contribute both to theoretical graph analysis and practical applications in molecular chemistry.
	\end{abstract}
	\section{Introduction}
	
	Graph theory has long served as a bridge between discrete mathematics and the molecular sciences, particularly through its applications in chemical graph theory. In computational chemistry, where the objective is to model and predict molecular properties without relying directly on quantum mechanical computations, molecular graphs provide a powerful abstraction. In such representations, atoms correspond to vertices, and chemical bonds to edges, often omitting hydrogen atoms for simplicity \cite{gutman1972graph, karelson1996quantum}. Within this framework, topological indices quantitative measures derived from the structure of molecular graphs are used to correlate molecular structure with chemical, physical, and biological properties.
	
	Among the earliest and most influential topological indices are the first and second Zagreb indices, denoted $\mathcal{M}_1(\mathbb{G})$ and $\mathcal{M}_2(\mathbb{G})$, respectively. Introduced by Gutman and Trinajsti\'{c} \cite{gutman1972graph}, these indices capture degree-based connectivity patterns and are defined for a graph $\mathbb{G} = (\mathcal{V}, \mathcal{E})$ as follows:
	\[
	\mathcal{M}_1(\mathbb{G}) = \sum_{v \in \mathcal{V}(\mathbb{G})} \zeta_\mathbb{G}(v)^2, \quad \mathcal{M}_2(\mathbb{G}) = \sum_{uv \in \mathcal{E}(\mathbb{G})} \zeta_\mathbb{G}(u) \zeta_\mathbb{G}(v),
	\]
	where $\zeta_\mathbb{G}(v)$ denotes the degree of vertex $v$ in $\mathbb{G}$. These indices have been widely applied in QSAR/QSPR studies for predicting molecular stability, boiling points, and reactivity \cite{todeschini2008handbook, nikolic2003zagreb}.
	
	In an effort to improve the predictive power of topological descriptors, particularly with respect to molecular branching and steric effects, Estrada et al. introduced the $ABC$ index \cite{estrada1998atom}, defined as:
	\[
	ABC(\mathbb{G}) = \sum_{uv \in \mathcal{E}(\mathbb{G})} \sqrt{\frac{\zeta_\mathbb{G}(u) + \zeta_\mathbb{G}(v) - 2}{\zeta_\mathbb{G}(u)\, \zeta_\mathbb{G}(v)}}.
	\]
	This index has demonstrated strong correlations with thermodynamic and structural properties, and has since become an important tool in chemical graph theory \cite{furtula2009atom}.
	
	A substantial body of work has been devoted to characterizing extremal values and establishing bounds for the $ABC$ and Zagreb indices in various graph classes. For trees, Furtula et al. \cite{furtula2009atom} showed that the star graph $\mathbb{S}_n$ attains the maximum $ABC$ index, a result further explored in chemical tree structures of bounded degree \cite{vassilev2012minimum}. Chen and Guo \cite{chen2011extreme} extended $ABC$ extremal studies to catacondensed hexagonal systems and showed that edge deletion reduces the $ABC$ index. Further contributions have identified $ABC$ extremal graphs among BFS trees \cite{lin2013minimal}, cactus graphs \cite{ashrafi2015extremal}, and graphs with specific connectivity or chromatic constraints \cite{chen2018solution, dimitrov2017remarks}. Most recently, Wu and Zhang \cite{wu2019structural}, Das et al. \cite{das2020maximal}, and Ali et al. \cite{ali2021atom} have deepened the analysis of structural conditions that maximize or minimize the $ABC$ index in chemical graphs.
	
	In parallel, extensive research has also been conducted on bounding the Zagreb indices under various structural constraints. Works such as those by Li and Zhou \cite{li2010maximum, li2011sharp}, Li and Zhang \cite{li2011sharp}, and Borovi\'{c}anin \cite{borovicanin2015extremal, borovicanin2016extremal} provide sharp bounds for Zagreb indices in graphs characterized by connectivity, matching numbers, branching vertices, and domination parameters. Further studies by Pei and Pan \cite{pei2018extremal}, Jia and Wang \cite{ji2018sharp}, Mojdeh et al. \cite{mojdeh2019zagreb}, and Enteshari and Taeri \cite{enteshari2021extremal} have extended these investigations to graphs with pendent vertices, cut vertices, and distance domination conditions. Dehgardi \cite{dehgardi2024lower} recently used Roman domination numbers to establish sharp lower bounds for Zagreb indices.
	
	Despite this rich literature, a notable gap remains the role of metric dimension a parameter measuring the minimum number of landmarks needed to uniquely identify every vertex in a graph in bounding degree-based topological indices has not been explored. The metric dimension provides insights into the graph's resolving capabilities and has practical implications in areas such as network navigation and robot localization.
	
	This work aims to address this gap by investigating the bounds on the Zagreb indices and the $ABC$ index for trees in terms of their order and metric dimension. Specifically, we establish upper and lower bounds for $\mathcal{M}_1$, $\mathcal{M}_2$, and an upper bound for the $ABC$ index. Furthermore, we identify the extremal trees that attain these bounds, thereby contributing both to the mathematical theory of graph invariants and to chemical graph modeling.

	\section{Preliminary} 
	In this section, we review key theoretical results and relevant definitions from the literature that underpin our approach. These foundational concepts will be essential in the development of the methodology and the proofs presented in next section.
	
	\begin{theorem}\textnormal{\label{preth1} \cite{gutman2004first}
			For a path graph \( \mathbb{P}_n \) with \( n \geq 3 \) vertices, the Zagreb indices \( \mathcal{M}_{1} \) and \( \mathcal{M}_{2} \) are given by the following formula:
			\[
			\mathcal{M}_{1}(\mathbb{P}_n) = 4n - 6\text{,\quad } \mathcal{M}_{2}(\mathbb{P}_n) = 2n - 8.
			\]}
	\end{theorem}
	
	\begin{theorem}\textnormal{\label{preth2} \cite{gutman2004first}
			For a star graph \( \mathbb{S}_n \) with \( n \geq 3 \) vertices, the Zagreb indices \( \mathcal{M}_{1} \) and \( \mathcal{M}_{2} \) are given by the following formula:
			\[
			\mathcal{M}_{1}(\mathbb{S}_n) = n(n-1) \text{,\quad } \mathcal{M}_{2}(\mathbb{S}_n) = (n-1)^{2}.
			\]}
	\end{theorem}

	\begin{definition}\textnormal{\cite{west2001introduction}}
		The collection of all vertices $ u \in \mathcal{V}(\mathbb{G}) $ such that $ v $ is next to $ u $ is the open neighborhood of a vertex $ v $, represented by $ N(v) $. In other words, 
		\[
		N(v) = \{ u \in \mathcal{V}(\mathbb{G}) \mid uv \in \mathcal{E}(\mathbb{G}) \}.
		\]
	\end{definition}
	
	\begin{definition}\textnormal{\cite{west2001introduction}}
		The degree of a vertex $ u $, denoted as $ \zeta_\mathbb{G}(u) $, represents the number of elements in the neighborhood $ N(x) $ of vertex $ u $.
	\end{definition}
	
	\begin{definition}\textnormal{\cite{west2001introduction}
			A vertex $ u $ is called a leaf if its degree, $ \zeta_\mathbb{G}(u)=1$. The diameter of a tree is defined as the greatest distance between any two leaf vertices, denoted by $ d $. A path $ \mathbb{P}_{d+1}: x_{1}, x_{2}, \ldots, x_{d+1} $ is called a diameter path in $\mathcal{T}$ if the path achieves this maximum distance.}
	\end{definition}
	
	\begin{definition}\textnormal{\cite{harary1976metric, slater1975leaves}
			For a connected graph  \(\mathbb{G} = (\mathcal{V}, \mathcal{E}) \), a subset \( \mathbb{S} \subseteq \mathcal{V} \) is referred to as a resolving set if, for every pair of distinct vertices \( u \) and \( v \) in \( \mathcal{V} \), there is a vertex \( s \in \mathbb{S} \) such that the distance from \( u \) to \( s \) differs from the distance from \( v \) to \( s \), i.e., \( d(u, s) \neq d(v, s) \). The metric dimension of the graph \( \mathbb{G} \) is defined as the smallest size of such a set \( \mathbb{S} \), denoted by $\varepsilon(\mathbb{G})$ }
	\end{definition}
	
	\begin{theorem}\textnormal{\label{preth3}\cite{chartrand2000resolvability}
			A connected graph $\mathbb{G}$ of order $n$ has $\varepsilon(\mathbb{G}))$=1 $\iff$ $\mathbb{G}$ is the path graph $\mathbb{P}_n$.}
	\end{theorem}
	\begin{theorem}\textnormal{\label{preth4} \cite{chartrand2000resolvability}
			For $n \geq 4$ vertices, let $\mathbb{S}_n$ be a star graph. Consequently, $\varepsilon(\mathbb{S}_n))=n-2$. Suppose we have $\mathbb{S}_n$ be a star graph with $n \geq 4$ vertices. Then, the $\varepsilon(\mathbb{S}_n))=n-2$. }
	\end{theorem}
	
	\begin{theorem}\textnormal{\label{preth1} \cite{furtula2009atom}
			Suppose that $\mathbb{S}_{n}$ is a star graph of order $n$ then $ABC(\mathbb{S}_{n})=\sqrt{n-2}\sqrt{n-1}$.}
	\end{theorem}

	\section{Upper Bound on $ABC$ index in terms of order and metric dimension}
	
	An extensive survey presented in \cite{furtula2009atom} explores the established bounds of the $ABC$ index across various fundamental graph classes. In particular, among all trees of order \( n \), the star graph \( \mathbb{S}_n \) is known to attain the maximum $ABC$ index. Building upon these findings, this section establishes new extremal bounds for the $ABC$ index in trees, incorporating both the order and the metric dimension as key structural parameters. Specifically, we define \(Max(n,\varepsilon(\mathcal{T})) \) as the sharp upper bound, which is rigorously derived and formally proven in Theorem \ref{mainth1}.

	\begin{lemma}\textnormal{\label{lem1}
			Let the function \( \mathfrak{\Upsilon(x)} \) be defined as $\mathfrak{\Upsilon(x)} = (\mathfrak{x}-1) \sqrt{\frac{\mathfrak{x}-1}{\mathfrak{x}}} - (\mathfrak{x}-2) \sqrt{\frac{\mathfrak{x}-2}{\mathfrak{x}-1}}, $ for \( \mathfrak{x} \geq 3 \). Then, \( \mathfrak{\Upsilon(x)} \) is an positive function.}
	\end{lemma}
	
	\begin{proof}
		Define \( h(\mathfrak{x}) \) as $h(\mathfrak{x}) = (\mathfrak{x}-1) \sqrt{\frac{\mathfrak{x}-1}{\mathfrak{x}}}.$ Computing its derivative, we obtain $h'(\mathfrak{x}) = \sqrt{\frac{\mathfrak{x}-1}{\mathfrak{x}}} + \frac{\mathfrak{x}-1}{2\mathfrak{x}^2} \sqrt{\frac{\mathfrak{x}}{\mathfrak{x}-1}}.$ Since \( h'(\mathfrak{x}) \) is positive for \( \mathfrak{x} \geq 2 \), it follows that \( h(\mathfrak{x}) \) is strictly increasing in this domain. Consequently, $\mathfrak{\Upsilon(x)} = h(\mathfrak{x}) - h(\mathfrak{x}-1) \geq 0,$ confirming that \( \mathfrak{\Upsilon(x)} \) is an positive function.
	\end{proof}

	\begin{lemma}\textnormal{\label{lem2}
			Let the function \( g(\mathfrak{x}) \) be defined as  
			$g(\mathfrak{x}) = \sqrt{\frac{\mathfrak{x}+ \mathfrak{y}-2}{\mathfrak{x} \mathfrak{y}}} - \sqrt{\frac{\mathfrak{x}+ \mathfrak{y}-3}{(\mathfrak{x}-1) \mathfrak{y}}},$ for \( \mathfrak{x} \geq 3 \) and any \( \mathfrak{y} \geq 2 \). Then, \( g(\mathfrak{x}) \) is a non positive function.}
	\end{lemma}
	
	\begin{proof}
		Differentiating \( g(\mathfrak{x}) \), we obtain $g'(\mathfrak{x}) = \frac{- \mathfrak{y} + 2}{2 \sqrt{\mathfrak{y}} \mathfrak{x}^{\frac{3}{2}} \sqrt{\mathfrak{x} + \mathfrak{y} - 2}} - \frac{- \mathfrak{y} + 2}{2 \sqrt{\mathfrak{y}} (\mathfrak{x}-1)^{\frac{3}{2}} \sqrt{\mathfrak{x} + \mathfrak{y} - 3}}.
		$ To establish that \( g'(\mathfrak{x}) \leq 0 \), we verify the inequality  $\frac{1}{\mathfrak{x}^{\frac{3}{2}} \sqrt{\mathfrak{x} + \mathfrak{y} - 2}} \geq \frac{1}{(\mathfrak{x}-1)^{\frac{3}{2}} \sqrt{\mathfrak{x} + \mathfrak{y} - 3}}.$ Define $p(\mathfrak{y}) = \frac{1}{\mathfrak{x}^{\frac{3}{2}} \sqrt{\mathfrak{x} + \mathfrak{y} - 2}}.$ Differentiating \( p(\mathfrak{y}) \), we get  $p'(\mathfrak{y}) = \frac{-1}{2\mathfrak{x}^{\frac{3}{2}}(\mathfrak{x} + \mathfrak{y} - 2)^{\frac{3}{2}}} < 0, \quad \text{for } \mathfrak{y} \geq 2.$ Since \( p(\mathfrak{y}) \) is a negative function, it follows that $\frac{1}{\mathfrak{x}^{\frac{3}{2}} \sqrt{\mathfrak{x} + \mathfrak{y} - 2}} < \frac{1}{(\mathfrak{x}-1)^{\frac{3}{2}} \sqrt{\mathfrak{x} + \mathfrak{y} - 3}}.$ Hence, \( g(\mathfrak{x}) \) is a non positive function for \( \mathfrak{x} \geq 3 \) and \( \mathfrak{y} \geq 2 \).
	\end{proof}
	
	\begin{lemma}\textnormal{\label{lem3}
			Let the function \( \digamma (\mathfrak{x}, \mathfrak{y}) \) be defined as $\digamma (\mathfrak{x}, \mathfrak{y}) = (\mathfrak{x}-1) \sqrt{\frac{\mathfrak{x}-1}{\mathfrak{x}}} + \sqrt{\frac{\mathfrak{x}+\mathfrak{y}-2}{\mathfrak{x} \mathfrak{y}}} - (\mathfrak{x}-2) \sqrt{\frac{\mathfrak{x}-2}{\mathfrak{x}-1}} - \sqrt{\frac{\mathfrak{x}+\mathfrak{y}-3}{(\mathfrak{x}-1) \mathfrak{y}}}.$ For all \( \mathfrak{x} \geq 3 \) and \( \mathfrak{y} \geq 2 \), the function satisfies the inequality $\digamma (\mathfrak{x}, \mathfrak{y}) > \frac{\sqrt{5}}{2\sqrt{2}}.$}
	\end{lemma}
	
	\begin{proof}
		We express \( \digamma (\mathfrak{x}, \mathfrak{y}) \) as the sum of two components: $\digamma_{1}(\mathfrak{x}) = (\mathfrak{x}-1) \sqrt{\frac{\mathfrak{x}-1}{\mathfrak{x}}} - (\mathfrak{x}-2) \sqrt{\frac{\mathfrak{x}-2}{\mathfrak{x}-1}}, $ $\digamma_{2}(\mathfrak{x}, \mathfrak{y}) = \sqrt{\frac{\mathfrak{x}+\mathfrak{y}-2}{\mathfrak{x} \mathfrak{y}}} - \sqrt{\frac{\mathfrak{x}+\mathfrak{y}-3}{(\mathfrak{x}-1) \mathfrak{y}}}.$ From Lemma \ref{lem1}, it follows that \( \digamma_{1}(\mathfrak{x}) \) is an positive function, and for \( \mathfrak{x} \geq 3 \), we have $\digamma_{1}(\mathfrak{x}) \geq 0.9258, \quad \text{with } \digamma_{1}(3) = 0.9258.$ Similarly, from Lemma \ref{lem2}, we know that \( \digamma_{2}(\mathfrak{x}, \mathfrak{y}) \) is a non positive function, and for \( \mathfrak{x} \geq 3 \) and \( \mathfrak{y} \geq 2 \), it satisfies $\digamma_{2}(\mathfrak{x}, \mathfrak{y}) \geq -0.1296.$ Since \( \digamma (\mathfrak{x}, \mathfrak{y}) = \digamma_1(\mathfrak{x}) + \digamma_2(\mathfrak{x}, \mathfrak{y}) \), we obtain $\digamma (\mathfrak{x}, \mathfrak{y}) \geq 0.7962.$ Noting that $0.7962 > \frac{\sqrt{5}}{2\sqrt{2}},$ we conclude that $\digamma (\mathfrak{x}, \mathfrak{y}) > \frac{\sqrt{5}}{2\sqrt{2}}.$ This completes the proof.
	\end{proof}

	$$Max(n,\varepsilon(\mathcal{T})) = \sqrt{n^{2}-3n+2}+\left(n-2-\varepsilon(\mathcal{T})\right)\left(\frac{4}{5}-\frac{2}{\sqrt{5}}\right).$$
	\begin{theorem}\textnormal{ \label{mainth1}
			Let $\mathcal{T}$ be a tree of order $n$ and $\varepsilon(\mathcal{T})$ is metric dimension of $\mathcal{T}$. Then the $ABC(\mathcal{T})\leq Max(n,\varepsilon(\mathcal{T}))$.}
	\end{theorem}
	\begin{proof}
		We claim that $ABC(\mathcal{T}) \leq Max(n, \varepsilon(\mathcal{T}))$ and prove this assertion by mathematical induction on $n$. For the base case $n = 4$, the two possible non-isomorphic graphs of order 4 are the path graph $\mathbb{P}_4$ and the star graph $\mathbb{S}_4$. A direct computation verifies that $ABC(\mathbb{S}_4) = Max(n, \varepsilon(\mathcal{T}))$ and $ABC(\mathbb{P}_4) < Max(n, \varepsilon(\mathcal{T}))$, establishing the base case. Now, assuming the inequality holds for all trees of order $n-1$, we proceed to show that it remains valid when the order of $T$ is $n$. We discuss in three cases.
		
		\noindent \textbf{Case 1:} Suppose that $\mathcal{T}$ has maximum degree $\Delta = n - 1$. This implies that $\mathcal{T} \cong \mathbb{S}_n$. By applying Theorems \ref{preth1} and \ref{preth4}, we obtain the following:
		
		\begin{equation*}
			\begin{split}
				ABC(\mathbb{S}_n) &= \sqrt{n-2}\sqrt{n-1}, \\
				&= \sqrt{n-2}\sqrt{n-1} + \left(n-2-\varepsilon(\mathcal{T})\right)\left(\frac{4}{5}-\frac{2}{\sqrt{5}}\right), \\
				&= \sqrt{n^2 - 3n + 2} + \left(n-2-\varepsilon(\mathcal{T})\right)\left(\frac{4}{5}-\frac{2}{\sqrt{5}}\right), \\
				&= Max(n, \varepsilon(\mathbb{S}_n)).
			\end{split}
		\end{equation*}
		
		\noindent Thus, Case 1 demonstrates that when $\mathcal{T} \cong \mathbb{S}_n$, we have $ABC(\mathbb{S}_n) = Max(n, \varepsilon(\mathbb{S}_n))$. Therefore, $\mathbb{S}_n$ is the tree that maximizes the $ABC$ index.
		
		\noindent \textbf{Case 2:} Consider the case where $\mathcal{T}$ has maximum degree $\Delta \geq 3$. Let $\mathbb{P}_{d+1} = \{v_1, v_2, \dots, v_{d+1}\}$ represent a path of diameter in $\mathcal{T}$, and assume that $\zeta_{\mathbb{G}}(v_2) \geq 2$.
		
		\noindent \textbf{Case 2.1:} Suppose that \( \zeta_{\mathbb{G}}(v_2) = 3 \) and \( \zeta_{\mathbb{G}}(v_3) = \Upsilon \geq 3 \). The neighborhood of $v_2$ is \( N(v_2) = \{v_1, v_3, m_1\} \), where \( \zeta_{\mathbb{G}}(m_1) = 1 \). Define the modified tree \( \mathcal{T}' =\mathcal{T} \setminus \{v_1\} \). In this case, the value of \( \varepsilon(\mathcal{T}) \) remains invariant, implying that \( \varepsilon(\mathcal{T}') = \varepsilon(\mathcal{T}) \). Consequently, we have:
		
		\begin{equation*}
			\begin{split}
				ABC(\mathcal{T}) =& ABC(\mathcal{T}') + \sqrt{\frac{\Upsilon+1}{3\Upsilon}}+2\left(\sqrt{\frac{2}{3}}-\sqrt{\frac{1}{2}}\right),\\
				\leq & \sqrt{(n-1)^{2}-3(n-1)+2}+\left(n-3-\varepsilon(T')\right)\left(\frac{4}{5}-\frac{2}{\sqrt{5}}\right)+ \sqrt{\frac{\Upsilon+1}{3\Upsilon}}+2\left(\sqrt{\frac{2}{3}}-\sqrt{\frac{1}{2}}\right),\\
				= & Max(n, \varepsilon(\mathcal{T})) +\sqrt{(n-1)^{2}-3(n-1)+2} -\sqrt{n^{2}-3n+2} + \sqrt{\frac{\Upsilon+1}{3\Upsilon}}- \left(\frac{4}{5}-\frac{2}{\sqrt{5}}\right)\\
				&+2\left(\sqrt{\frac{2}{3}}-\sqrt{\frac{1}{2}}\right).\\
			\end{split}
		\end{equation*}
		
		\noindent Suppose that the function \( f(n, \Upsilon) \) is defined as  
		$f(n, \Upsilon) = \sqrt{(n-1)^{2} - 3(n-1) + 2} - \sqrt{n^{2} - 3n + 2} + \sqrt{\frac{\Upsilon+1}{3\Upsilon}} - \left(\frac{4}{5} - \frac{2}{\sqrt{5}}\right) + 2\left(\sqrt{\frac{2}{3}} - \sqrt{\frac{1}{2}}\right)$. We claim that \( f(n, \Upsilon) \) is a negative function. To establish this, we decompose \( f(n, \Upsilon) \) into two separate functions: $f_1(n) = \sqrt{(n-1)^{2} - 3(n-1) + 2} - \sqrt{n^{2} - 3n + 2}$, $f_2(\Upsilon) = \sqrt{\frac{\Upsilon+1}{3\Upsilon}} - \left(\frac{4}{5} - \frac{2}{\sqrt{5}}\right) + 2\left(\sqrt{\frac{2}{3}} - \sqrt{\frac{1}{2}}\right)$. First, we analyze \( f_1(n) \). By rewriting it in terms of the function \( f_{1'}(n) = \sqrt{n^{2} - 3n + 2} \), we compute its derivative:  $\frac{d}{dn} f_{1'}(n) = \frac{2n - 3}{2\sqrt{n^{2} - 3n + 2}}$. Since this derivative is positive for \( n \geq 5 \), it follows that \( f_{1'}(n) \) is an increasing function. Consequently, \( f_1(n) \) is negative and bounded within the interval \( (-1.015, -1) \).  Next, we consider \( f_2(\Upsilon) \). Computing its derivative, we obtain: $\frac{d}{d\Upsilon} f_2(\Upsilon) = -\frac{1}{2\sqrt{3} \Upsilon^{2} \sqrt{\frac{\Upsilon+1}{\Upsilon}}}$. Since this derivative is negative, \( f_2(\Upsilon) \) is a decreasing function and is bounded within the interval \( (0.89, 0.98) \).  Thus, summing both components, we conclude that $f(n, \Upsilon) = f_1(n) + f_2(\Upsilon)$ is negative, as required.
		Therefore, 
		
		\begin{equation*}
			\begin{split}
				ABC(\mathcal{T}) = & Max(n, \varepsilon(\mathcal{T})) + f(n, \Upsilon)<Max(n,\varepsilon(\mathcal{T})).                           
			\end{split}
		\end{equation*}
		
		\noindent \textbf{Case 2.1.1:} Suppose that \( \zeta_{\mathbb{G}}(v_3) =  2 \). Define the modified tree \( \mathcal{T}' = \mathcal{T} \setminus \{v_1, m_1\} \). In this case, the value of \( \varepsilon(\mathcal{T}) \) remains invariant, implying that \( \varepsilon(T') = \varepsilon(\mathcal{T}) \). Consequently, we have:
		
		\begin{equation*}
			\begin{split}
				ABC(\mathcal{T}) = & ABC(\mathcal{T}') + 2\sqrt{\frac{2}{3}},\\
				\leq &  \sqrt{(n-2)^{2}-3(n-2)+2}+\left(n-4-\varepsilon(T')\right)\left(\frac{4}{5}-\frac{2}{\sqrt{5}}\right)+2\sqrt{\frac{2}{3}},\\
				= & Max(n, \varepsilon(\mathcal{T})) +\sqrt{(n-2)^{2}-3(n-2)+2} -\sqrt{n^{2}-3n+2} -2 \left(\frac{4}{5}-\frac{2}{\sqrt{5}}\right)+ 2\sqrt{\frac{2}{3}}.\\
			\end{split}
		\end{equation*}
		
		\noindent Suppose that $f(n)=\sqrt{(n-2)^{2}-3(n-2)+2} -\sqrt{n^{2}-3n+2} -2 \left(\frac{4}{5}-\frac{2}{\sqrt{5}}\right)+ 2\sqrt{\frac{2}{3}}$, and $f(n)$ is negative function for $n \geq 5$. Therefore, 
		\begin{equation*}
			\begin{split}
				ABC(\mathcal{T}) = & Max(n, \varepsilon(\mathcal{T})) + f(n) <Max(n, \varepsilon(\mathcal{T})).\\
			\end{split}
		\end{equation*}
		
		\noindent\textbf{Case 2.2:} Assume that $\zeta_{\mathbb{G}}(v_2) = \psi \geq 4$ and $\zeta_{\mathbb{G}}(v_{3})=\tau$. The neighborhood of $v_2$ is given by $N(v_2) = \{v_1, v_3, m_1,\ldots,m_{\psi-2}\}$, where $\zeta_{\mathbb{G}}(m_i) = 1$ and $1\leq i \leq \psi-2$. Now, let $T' = T \setminus \{v_1\}$ represent the modified tree. In this case, the value of $\varepsilon(\mathcal{T})$ changed, meaning $\varepsilon(T') = \varepsilon(\mathcal{T})-1$. Therefore, we have:
		
		\begin{equation*}
			\begin{split}
				ABC(\mathcal{T}) = & ABC(\mathcal{T}') + \frac{(\psi-1)^{\frac{3}{2}}}{\sqrt{\psi}}-\frac{(\psi-2)^{\frac{3}{2}}}{\sqrt{\psi -1}}+\sqrt{\frac{\psi+\tau-2}{\psi l}}-\sqrt{\frac{\psi+\tau-3}{(\psi -1) \tau}},\\
				\leq & \sqrt{(n-1)^{2}-3(n-1)+2}+\left(n-3-\varepsilon(T')\right)\left(\frac{4}{5}-\frac{2}{\sqrt{5}}\right) +  \frac{(\psi-1)^{\frac{3}{2}}}{\sqrt{\psi}}+\sqrt{\frac{\psi+\tau-2}{\psi l}}\\
				&-\frac{(\psi-2)^{\frac{3}{2}}}{\sqrt{\psi -1}}-\sqrt{\frac{\psi+\tau-3}{(\psi -1) \tau}},\\
				= & Max(n, \varepsilon(\mathcal{T}))-\sqrt{n^{2}-3n+2} + \sqrt{(n-1)^{2}-3(n-1)+2}  + (\psi -1)\sqrt{\frac{\psi-1}{\psi}}\\&+\sqrt{\frac{\psi+\tau-2}{\psi l}}
				-(\psi -2)\sqrt{\frac{\psi-2}{\psi -1}}-\sqrt{\frac{\psi+\tau-3}{(\psi-1) \tau}}.\\
			\end{split}
		\end{equation*}
		
		\noindent Consider the function  $f(n) =  - \sqrt{n^{2} - 3n + 2}+\sqrt{(n-1)^{2} - 3(n-1) + 2}$ and define  $g(\psi, \tau) = (\psi -1)\sqrt{\frac{\psi-1}{\psi}} + \sqrt{\frac{\psi+\tau-2}{\psi \tau}} - (\psi -2)\sqrt{\frac{\psi-2}{\psi -1}} - \sqrt{\frac{\psi+\tau-3}{(\psi-1) \tau}}$. From the result established in Case 2.1, we have $-1.015 < f(n) < -1$. Additionally, by applying Lemma \ref{lem3}, it follows that $0.796 < g(\psi, \tau) < 1$. Since \( f(n) \) is negative and \( g(\psi, \tau) \) remains positive, their sum $w(n, \psi, \tau) = f(n) + g(\psi, \tau)$ is negative. Therefore, the function \( w(n, \psi, \tau) \) is strictly less than zero. Hence,
		
		\begin{equation*}
			\begin{split}
				ABC(\mathcal{T}) = & Max(n, \varepsilon(\mathcal{T})) + w(n, \psi, \tau) <Max(n, \varepsilon(\mathcal{T})).\\
			\end{split}
		\end{equation*}
		
		\noindent \textbf{Case 2.3:} Suppose that \( \zeta_{\mathbb{G}}(v_2) = 2 \) and \( \zeta_{\mathbb{G}}(v_3) \geq 2 \).  
		
		\noindent \textbf{Case 2.3.1:} Consider the case where \( \zeta_{\mathbb{G}}(v_3) = \psi \), where $\psi=\{2,3\}$, with \( N(v_3) = \{v_2, v_4, v_\lambda\} \), where \( \lambda \leq \psi - 2 \), \( \zeta_{\mathbb{G}}(v_\lambda) = 1 \), and $\zeta_{\mathbb{G}}(v_{4})=\tau$. Define \( T' = T \setminus \{v_1, v_2\} \). Under these conditions, the value of \( \varepsilon(\mathcal{T}) \) remains unchanged, implying that \( \varepsilon(T') = \varepsilon(\mathcal{T}) \). Therefore, we obtain:
		
		\begin{equation*}
			\begin{split}
				ABC(\mathcal{T}) = & ABC(\mathcal{T}') + \sqrt{\frac{\psi-1}{\psi}}-\sqrt{\frac{\psi+\tau-3}{(\psi-1)\tau}}+\sqrt{\frac{\psi +\tau-2}{\psi \tau}}+\sqrt{\frac{1}{2}},\\
				\leq & \sqrt{(n-2)^{2}-3(n-2)+2}+\left(n-4-\varepsilon(T')\right)\left(\frac{4}{5}-\frac{2}{\sqrt{5}}\right) + \sqrt{\frac{\psi-1}{\psi}}+\sqrt{\frac{\psi +\tau-2}{\psi \tau}}\\
				&-\sqrt{\frac{\psi+\tau-3}{(\psi-1)\tau}}+\sqrt{\frac{1}{2}},\\
				= & Max(n, \varepsilon(\mathcal{T}))-\sqrt{n^{2}-3n+2} + \sqrt{(n-2)^{2}-3(n-2)+2} + \sqrt{\frac{\psi-1}{\psi}}+\sqrt{\frac{\psi +\tau-2}{\psi \tau}}\\
				&-\sqrt{\frac{\psi+\tau-3}{(\psi-1)\tau}}+\sqrt{\frac{1}{2}}-2\left(\frac{4}{5}-\frac{2}{\sqrt{5}}\right).\\
			\end{split}
		\end{equation*}
		
		\noindent Consider the functions \( f(n) =- \sqrt{n^{2} - 3n + 2}+ \sqrt{(n-2)^{2} - 3(n-2) + 2}  \) and \( g(\psi, \tau) = \sqrt{\frac{\psi + \tau - 2}{\psi \tau}} - \sqrt{\frac{\psi + \tau - 3}{(\psi-1) \tau}} + \sqrt{\frac{1}{2}} - 2\left(\frac{4}{5} - \frac{2}{\sqrt{5}}\right) \). It is straightforward to verify that \( f(n) \) is a negative function, with values constrained within the interval \( (-2.0226, -2) \) for \( n \geq 6 \). Additionally, \( g(\psi, \tau) \) is a positive function, bounded within \( (1.56, 1.72) \). Since their sum remains negative, it follows that \( w(n, \psi, \tau) = f(n) + g(\psi, \tau) < 0 \). Consequently, we conclude that:
		
		\begin{equation*}
			\begin{split}
				ABC(\mathcal{T}) = & Max(n, \varepsilon(\mathcal{T})) + w(n, \psi, \tau) <Max(n, \varepsilon(\mathcal{T})).\\
			\end{split}
		\end{equation*}  
		
		\noindent \textbf{Case 2.3.2:} Suppose that \( \zeta_{\mathbb{G}}(v_3) = \psi  \geq 4 \), where the neighborhood of \( v_3 \) is given by \( N(v_3) = \{v_2, v_4, \tau_1, \dots, \tau_{\psi-2}\} \), and assume that \( \zeta_{\mathbb{G}}(\tau_{\psi-2}) = 1 \) and $\zeta_{\mathbb{G}}(v_4) = \tau$. Define \( T' = T \setminus \{v_1, v_2, \tau_1\} \). In this scenario, the value of \( \varepsilon(\mathcal{T}) \) decreases by 2, leading to \( \varepsilon(T') = \varepsilon(\mathcal{T}) - 2 \). Hence, we obtain:
		
		\begin{equation*}
			\begin{split}
				ABC(\mathcal{T}) = & ABC(\mathcal{T}') + (\psi-2)\sqrt{\frac{\psi-1}{\psi}}-(\psi-3)\sqrt{\frac{\psi+1}{\psi-2}}+\sqrt{\frac{\psi+\tau-2}{\psi \tau}} -\sqrt{\frac{\psi +\tau -4}{(\psi -2)\tau}}+\sqrt{2},\\
				\leq & \sqrt{(n-3)^{2}-3(n-3)+2}+\left(n-3-\varepsilon(\mathcal{T})\right)\left(\frac{4}{5}-\frac{2}{\sqrt{5}}\right)+ (\psi-2)\sqrt{\frac{\psi-1}{\psi}}\\ &+\sqrt{\frac{\psi+\tau-2}{\psi \tau}}-(\psi-3)\sqrt{\frac{\psi+1}{\psi-2}}-\sqrt{\frac{\psi +\tau -4}{(\psi -2)\tau}}+\sqrt{2},\\
				= & Max(n, \varepsilon(\mathcal{T}))-\sqrt{n^{2}-3n+2} + \sqrt{(n-3)^{2}-3(n-3)+2} + (\psi-2)\sqrt{\frac{\psi-1}{\psi}}\\ &+\sqrt{\frac{\psi+\tau-2}{\psi \tau}}-(\psi-3)\sqrt{\frac{\psi+1}{\psi-2}}-\sqrt{\frac{\psi +\tau -4}{(\psi -2)\tau}}+\sqrt{2} -\left(\frac{4}{5}-\frac{2}{\sqrt{5}}\right).\\
			\end{split}
		\end{equation*}
		
		\noindent Let \( f(n) = - \sqrt{n^{2} - 3n + 2}+\sqrt{(n-3)^{2} - 3(n-3) + 2}  \) and $g(\psi, \tau) = (\psi -2) \sqrt{\frac{\psi -1}{\psi}} + \sqrt{\frac{\psi + \tau -2}{\psi \tau}} - (\psi -3) \sqrt{\frac{\psi +1}{\psi -2}} - \sqrt{\frac{\psi + \tau -4}{(\psi -2) \tau}} + \sqrt{2} - \left(\frac{4}{5} - \frac{2}{\sqrt{5}}\right).$ It is straightforward to verify that \( f(n) \) is a negative function, with values restricted to the interval \( (-3.05, -3) \) for \( n \geq 6 \). Furthermore, \( g(\psi, \tau) \) is a positive function, bounded within \( (0.50, 1.66) \). Since their sum remains negative, it follows that $w(n, \psi, \tau) = f(n) + g(\psi, \tau) < 0$. Thus, we conclude that:
		
		\begin{equation*}
			\begin{split}
				ABC(\mathcal{T}) = & Max(n, \varepsilon(\mathcal{T})) + w(n, \psi, \tau) <Max(n, \varepsilon(\mathcal{T})).\\
			\end{split}
		\end{equation*}  
		
		\noindent Hence, the theorem is proven.
	\end{proof}

	\section{Bounds on Zagreb indices in terms of order and metric dimension}
	In this part, we establish bounds on the Zagreb indices of trees with respect to their order and metric dimension, and characterize the tree structures that attain these extremal values. By analyzing the interplay between the Zagreb indices and the metric dimension, we identify specific tree configurations that either maximize or minimize the indices under given constraints, offering insight into their structural properties and behavior in relation to these graph invariants.
	
	\begin{theorem}\label{mainthm1}
		\textnormal{Assume that $\mathcal{T}$ be a tree with order $n$ and metric dimension $\varepsilon(\mathcal{T})$. Then, we have} $$\mathcal{M}_{1}(\mathcal{T})\geq 4n-7+\varepsilon(\mathcal{T}).$$
		
	\end{theorem}
	\begin{proof} Suppose that $g_{min}(n, \varepsilon(\mathcal{T}))=4n-7+\varepsilon(\mathcal{T})$. 
		We shall use mathematical induction of the order $n$ to establish the result. If $n\geq 4$, then by using Theorems \ref{preth1}, and \ref{preth3} $\mathcal{M}_{1}(\mathbb{P}_{4})=10=g_{min}(4,1)$ and using Theorems \ref{preth2}, and \ref{preth4} $\mathcal{M}_{1}(\mathbb{S}_{4})=12\geq g_{min}(4,2)$. Let us now establish the result when the tree has order $n$, on the assumption that any tree with order $n-1$ will satisfy the statement. Thus, we now demonstrate Theorem \ref{mainthm1} in the two cases that follow: 
		
		\noindent \textbf{Case 1:} Assume that we take $\Delta(\mathcal{T}) = 2$, implying that $\mathcal{T} \cong \mathbb{P}_n$. By applying Theorem \ref{preth1}, $\mathcal{M}_{1}(\mathbb{P}_n)$ is obtained, and using Theorem \ref{preth3}, we find the value of $\varepsilon(\mathbb{P}_n)$:
		\[
		\begin{split}
			\mathcal{M}_{1}(\mathbb{P}_n) &= 4n - 6, \\
			&= 4n - 6 + \varepsilon(\mathbb{P}_n) - 1, \\
			&= 4n - 7 + \varepsilon(\mathbb{P}_n),\\
			&=g_{min}(n, \varepsilon(\mathcal{T})).
		\end{split}
		\]
		
		Thus, equality holds when $\mathcal{T} \cong \mathbb{P}_n$. Consequently, we conclude that if $\mathcal{T} \cong \mathbb{P}_n$, then $\mathcal{M}_{1}(\mathcal{T})$ attains its minimum value.
		
		\noindent \textbf{Case 2:} Let $ \mathbb{P}_{d+1} = \{\mathbb{S}_1, \mathbb{S}_2, \ldots, \mathbb{S}_{d+1}\} $ be a diametrical path of the tree $\mathcal{T}$, and suppose that $ \Delta(\mathcal{T}) $, is at least 3. Furthermore, $\zeta_\mathcal{T}(\mathbb{S}_2) \geq 2$.
		
		\noindent \textbf{Case 2.1:} We assume that $\zeta_\mathcal{T}(\mathbb{S}_2) = u$, where $u =\{2, 3\}$ and the neighborhood of $\mathbb{S}_2$ is $N(\mathbb{S}_2) = \{\mathbb{S}_1, \mathbb{S}_3, t_1, t_{2}\}$. If we consider the modified tree $\mathcal{T}' = \mathcal{T} - \{\mathbb{S}_1\}$, then the value of $\varepsilon(\mathcal{T})$ remains unchanged, i.e., $\varepsilon(\mathcal{T}') = \varepsilon(\mathcal{T})$. We then have:
		\[
		\begin{split}
			\mathcal{M}_{1}(\mathcal{T}) &= \mathcal{M}_{1}(\mathcal{T}') + 2u, \\
			&\geq 4(n-1) - 7 + \varepsilon(\mathcal{T}) + 2u , \\
			&= g_{\min}(n, \varepsilon(\mathcal{T})) + 2u - 4, \\
			&\geq g_{\min}(n, \varepsilon(\mathcal{T})).
		\end{split}
		\]
		
		\noindent \textbf{Case 2.2:} We assume that $\zeta_\mathcal{T}(\mathbb{S}_2) = v \geq 4$, and the neighborhood of $\mathbb{S}_2$ is given by $N(\mathbb{S}_2) = \{\mathbb{S}_1, \mathbb{S}_3, t_1, \ldots, t_{v-2}\}$. If we consider the modified tree $\mathcal{T}' = \mathcal{T} - \{\mathbb{S}_1\}$, then the value of $\varepsilon(\mathcal{T})$ changes, and we have $\varepsilon(\mathcal{T}') = \varepsilon(\mathcal{T}) - 1$. Thus, we arrive at:
		
		\[
		\begin{split}
			\mathcal{M}_{1}(\mathcal{T})=&\mathcal{M}_{1}(\mathcal{T}')+2v,\\
			\geq & 4(n-1) - 7 + \varepsilon(\mathcal{T})-1+2v,\\
			=& g_{\min}(n, \varepsilon(\mathcal{T}))+2v-5,\\
			>& g_{\min}(n, \varepsilon(\mathcal{T})).
		\end{split}
		\]
		
		Hence, the theorem is proved.
	\end{proof}
	
	\begin{theorem} \textnormal{\label{mainthm2}
			Assume that $\mathcal{T}$ be a tree with order $n$ and metric dimension $\varepsilon(\mathcal{T})$. The First Zagreb Index $\mathcal{M}_{1}(\mathcal{T})$ of $\mathcal{T}$ satisfies the following upper bound:
			\[
			\mathcal{M}_{1}(\mathcal{T}) \leq n + (n-1)(n-2) + \varepsilon(\mathcal{T}),
			\]
			where $g(n, \varepsilon(\mathcal{T}))$ is a function determined by the order $n$ of the tree and its metric dimension $\varepsilon(\mathcal{T})$.}
	\end{theorem}
	\begin{proof}
		Assume that $g_{\max}(n, \varepsilon(\mathcal{T})) = n + (n-1)(n-2) + \varepsilon(\mathcal{T})$. We will prove this result by induction on the order $n$. For $n \geq 4$, applying Theorems \ref{preth2} and \ref{preth4}, we find that $\mathcal{M}_{1}(\mathbb{P}_4) = 11 < g_{\min}(4, 1)$, and similarly, $\mathcal{M}_{1}(\mathbb{S}_4) = 12 = g_{\min}(4, 2)$. Now, suppose the statement holds for any $\mathcal{T}$ with order $n-1$, and we aim to establish it for a $\mathcal{T}$ with order $n$. Thus, we now demonstrate Theorem \ref{mainthm2} in the two cases that follow:
		
		\noindent \textbf{Case 1:} Assume $\Delta(\mathcal{T}) = n-1$, then we have $\mathcal{T} \cong \mathbb{S}_n$. Using Theorem \ref{preth2}, we can determine $\mathcal{M}_{1}(\mathbb{S}_n)$, and with Theorem \ref{preth4}, we compute $\varepsilon(\mathbb{S}_n)$:
		\[
		\begin{split}
			\mathcal{M}_{1}(\mathbb{S}_n) &= n(n-1), \\
			&= n(n-1) + \varepsilon(\mathbb{S}_n) - (n-2), \\
			&= n^2 - 2n + 2 + \varepsilon(\mathbb{S}_n), \\
			&= n + (n-1)(n-2) + \varepsilon(\mathbb{S}_n), \\
			&= g_{\max}(n, \varepsilon(\mathbb{S}_n)).
		\end{split}
		\]
		Thus, equality is achieved when $\mathcal{T} \cong \mathbb{S}_n$. Therefore, we conclude that when $T \cong \mathbb{S}_n$, $\mathcal{M}_{1}(\mathcal{T})$ reaches its maximum value.
		
		\noindent \textbf{Case 2:} Let $ \mathbb{P}_{d+1} = \{\mathbb{S}_1, \mathbb{S}_2, \dots, \mathbb{S}_{d+1}\} $ represent a diametrical path in the tree $\mathcal{T}$, where $ \Delta(\mathcal{T}) \geq 3 $. Additionally, assume $ \zeta_\mathcal{T}(\mathbb{S}_2) \geq 2 $.
		
		\noindent \textbf{Case 2.1:} Suppose $ \zeta_\mathcal{T}(\mathbb{S}_2) = u $, where $ u \in \{2, 3\} $, with the order of $\mathcal{T}$ at least 5. The neighborhood of $ \mathbb{S}_2 $ is given by $ N(\mathbb{S}_2) = \{\mathbb{S}_1, \mathbb{S}_3, t_1, t_2\} $. If we define a modified tree $ \mathcal{T}' = \mathcal{T} - \{\mathbb{S}_1\} $, the value of $ \varepsilon(\mathcal{T}) $ remains constant, i.e., $ \varepsilon(\mathcal{T}') = \varepsilon(\mathcal{T}) $. Therefore, we can express the following:
		\[
		\begin{aligned}
			\mathcal{M}_{1}(\mathcal{T}) &= \mathcal{M}_{1}(\mathcal{T}') + 2u, \\
			&\leq n-1 + (n-2)(n-3) + \varepsilon(\mathcal{T}) + 2u, \\
			&= g_{\max}(n, \varepsilon(\mathcal{T})) - 2n + 2u + 3, \\
			&< g_{\max}(n, \varepsilon(\mathcal{T})).
		\end{aligned}
		\]
		
		\noindent \textbf{Case 2.2:} Let us assume that $ \zeta_\mathcal{T}(\mathbb{S}_2) = v \geq 4 $, where the neighborhood of $ \mathbb{S}_2 $ is defined as $ N(\mathbb{S}_2) = \{\mathbb{S}_1, \mathbb{S}_3, t_1, \dots, t_{v-2}\} $. If we modify the tree to $ \mathcal{T}' = \mathcal{T} - \{\mathbb{S}_1\} $, the value of $ \varepsilon(\mathcal{T}) $ changes, resulting in $ \varepsilon(\mathcal{T}') = \varepsilon(\mathcal{T}) - 1 $. Consequently, we have:
		\[
		\begin{aligned}
			\mathcal{M}_{1}(\mathcal{T}) &= \mathcal{M}_{1}(\mathcal{T}') + 2v, \\
			&\leq n^{2}-2n+2 + \varepsilon(\mathcal{T})-2n+2 + 2v, \\
			&=  g_{\max}(n, \varepsilon(\mathcal{T})) - 2n + 2v + 2, \\
			&< g_{\max}(n, \varepsilon(\mathcal{T})).
		\end{aligned}
		\]
		Thus, the proof is complete.
	\end{proof}
	
	\begin{theorem}\textnormal{\label{mainthm3}
			Suppose $\mathcal{T}$ is any tree of order $n$ with metric dimension $\varepsilon(\mathcal{T})$. Then, $$ \mathcal{M}_{2}(\mathcal{T}) \geq 4n + \varepsilon(\mathcal{T})-9.$$
		}
	\end{theorem}
	\begin{proof}
		Assume that $ g_{min}(n, \varepsilon(\mathcal{T})) = 4n - 9 + \varepsilon(\mathcal{T}) $. Induction on the order $n$ will be used to demonstrate the result. For the base case where $ n \geq 4 $, by applying Theorems \ref{preth1} and \ref{preth3}, we know that $ \mathcal{M}_{2}(\mathbb{P}_4) = 8 = g_{min}(4, 1) $. Similarly, by using Theorems \ref{preth2} and \ref{preth4}, we find that $ \mathcal{M}_{1}(\mathbb{S}_4) = 9 \geq g_{min}(4, 2) $.
		
		Assuming that the assertion is true for every tree of order $ n-1 $, our goal is to demonstrate it for a tree of order $ n $. The following scenarios will comprise the proof:
		
		\noindent \textbf{Case 1:} Assume $ \Delta(\mathcal{T}) = 2 $, which implies that $ \mathcal{T} \cong \mathbb{P}_n $. By Theorem \ref{preth1}, we calculate $ \mathcal{M}_{1}(\mathbb{P}_n) $, and by Theorem \ref{preth3}, we determine $ \varepsilon(\mathbb{P}_n) $:
		
		\[
		\begin{split}
			\mathcal{M}_{2}(\mathbb{P}_n) &= 4n - 8, \\
			&= 4n - 8 + \varepsilon(\mathbb{P}_n) - 1, \\
			&= 4n - 9 + \varepsilon(\mathbb{P}_n), \\
			&= g_{min}(n, \varepsilon(\mathcal{T})).
		\end{split}
		\]
		
		Thus, equality is achieved when $ \mathcal{T} \cong \mathbb{P}_n $. Therefore, we conclude that when $ \mathcal{T} \cong \mathbb{P}_n $, the value of $ \mathcal{M}_{2}(\mathcal{T}) $ reaches its minimum.
		
		\noindent \textbf{Case 2:} Let $ \mathbb{P}_{d+1} = \{\mathbb{S}_1, \mathbb{S}_2, \ldots, \mathbb{S}_{d+1}\} $ denote a diametrical path in the tree $\mathcal{T}$. We suppose that the greatest degree $ \Delta(\mathcal{T}) $ is at least 3, with the constraint that $ \zeta_\mathcal{T}(\mathbb{S}_2) \geq 2 $.
		
		\noindent \textbf{Subcase 2.1:} Suppose $ \zeta_\mathcal{T}(\mathbb{S}_2) = u $ where $ u \in \{2, 3\} $, and let the neighborhood of $ \mathbb{S}_2 $ be $ N(\mathbb{S}_2) = \{\mathbb{S}_1, \mathbb{S}_3, t_1, t_2\} $. If we construct the modified tree $ \mathcal{T}' = \mathcal{T} - \{\mathbb{S}_1\} $, then the metric dimension remains unchanged, i.e., $ \varepsilon(\mathcal{T}') = \varepsilon(\mathcal{T}) $. Consequently, we have:
		
		\[
		\begin{split}
			\mathcal{M}_{2}(\mathcal{T}) &= \mathcal{M}_{2}(\mathcal{T}') + 2u, \\
			&\geq 4(n-1) - 9 + \varepsilon(\mathcal{T}) + 2u, \\
			&= g_{\min}(n, \varepsilon(\mathcal{T})) + 2u - 4, \\
			&\geq g_{\min}(n, \varepsilon(\mathcal{T})).
		\end{split}
		\]
		
		\noindent \textbf{Subcase 2.2:} Now, assume $ \zeta_\mathcal{T}(\mathbb{S}_2) = v $ where $ v \geq 4 $, with the neighborhood given by $ N(\mathbb{S}_2) = \{\mathbb{S}_1, \mathbb{S}_3, t_1, \ldots, t_{v-2}\} $. In this case, if we define $ \mathcal{T}' = \mathcal{T} - \{\mathbb{S}_1\} $, the metric dimension decreases, resulting in $ \varepsilon(\mathcal{T}') = \varepsilon(\mathcal{T}) - 1 $. Thus, we can write:
		
		\[
		\begin{split}
			\mathcal{M}_{2}(\mathcal{T}) &= \mathcal{M}_{2}(\mathcal{T}') + 2v, \\
			&\geq 4(n-1) - 9 + \varepsilon(\mathcal{T}) - 1 + 2v, \\
			&= g_{\min}(n, \varepsilon(\mathcal{T})) + 2v - 5, \\
			&> g_{\min}(n, \varepsilon(\mathcal{T})).
		\end{split}
		\]
		
		Therefore, the theorem is established.

	\end{proof}
	
	\begin{theorem}\textnormal{\label{mainthm4}
			Let $\mathcal{T}$ be a tree with order $ n $ and metric dimension $ \varepsilon(\mathcal{T}) $. Then
			$$ \mathcal{M}_{2}(\mathcal{T}) \leq n^2 - 3n + \varepsilon(\mathcal{T})+ 3 . $$}
	\end{theorem}
	\begin{proof}
		Assume that $ g_{\max}(n, \varepsilon(\mathcal{T})) = n^2 - 3n + \varepsilon(\mathcal{T})+ 3 $. We will demonstrate this result using mathematical induction based on the order $ n $. For $ n \geq 4 $, by applying Theorems \ref{preth2} and \ref{preth4}, we have $ \mathcal{M}_{2}(\mathbb{P}_4) = 8 < g_{\max}(4, 1) $ and $ \mathcal{M}_{2}(\mathbb{S}_4) = 9 = g_{\max}(4, 2) $. 
		
		Assuming that the statement holds true for any $\mathcal{T}$ with order $ n-1 $, we will now establish it for a $\mathcal{T}$ with order $ n $. In the next two cases, Theorem \ref{mainthm4} will be proved:
		
		\noindent \textbf{Case 1:} Suppose $ \Delta(\mathcal{T}) = n-1 $, which indicates that $ \mathcal{T} \cong \mathbb{S}_n $. By utilizing Theorem \ref{preth2}, we find $ \mathcal{M}_{2}(\mathbb{S}_n) $, and with Theorem \ref{preth4}, we determine $ \varepsilon(\mathbb{S}_n) $:
		
		\[
		\begin{split}
			\mathcal{M}_{2}(\mathbb{S}_n) &= (n-1)^{2}, \\
			&= (n-1)^{2} + \varepsilon(\mathbb{S}_n) - (n-2), \\
			&= n^2 - 2n + 1 + \varepsilon(\mathbb{S}_n)-n+2, \\
			&= n^2 - 3n + \varepsilon(\mathbb{S}_n)+ 3 , \\
			&= g_{\max}(n, \varepsilon(\mathbb{S}_n)).
		\end{split}
		\]
		
		Consequently, equality holds when $ \mathcal{T} \cong \mathbb{S}_n $. Therefore, we conclude that if $ \mathcal{T} \cong \mathbb{S}_n $, then $ \mathcal{M}_{2}(\mathcal{T}) $ attains its maximum value.
		
		\noindent \textbf{Case 2:} Let $ \mathbb{P}_{d+1} = \{\mathbb{S}_1, \mathbb{S}_2, \dots, \mathbb{S}_{d+1}\} $ denote a diametrical path within the tree $\mathcal{T}$, where $ \Delta(\mathcal{T}) \geq 3 $. Furthermore, assume $ \zeta_\mathcal{T}(\mathbb{S}_2) \geq 2 $.
		
		\noindent \textbf{Case 2.1:} Assume $ \zeta_\mathcal{T}(\mathbb{S}_2) = u $, where $ u \in \{2, 3\} $, and the order of tree is at least 5. The neighborhood of $ \mathbb{S}_2 $ can be represented as $ N(\mathbb{S}_2) = \{\mathbb{S}_1, \mathbb{S}_3, t_1, t_2\} $. By defining the modified tree $ \mathcal{T}' = \mathcal{T} - \{\mathbb{S}_1\} $, we observe that the value of $ \varepsilon(\mathcal{T}) $ remains unchanged, such that $ \varepsilon(\mathcal{T}') = \varepsilon(\mathcal{T}) $. Hence, we can express:
		
		\[
		\begin{aligned}
			\mathcal{M}_{2}(\mathcal{T}) &= \mathcal{M}_{2}(\mathcal{T}') + 2u, \\
			&\leq  (n-1)^2 - 3(n-1) + \varepsilon(\mathcal{T})+ 3 + 2u, \\
			&= g_{\max}(n, \varepsilon(\mathcal{T})) - 2n + 2u + 4, \\
			&\leq g_{\max}(n, \varepsilon(\mathcal{T})).
		\end{aligned}
		\]
		
		\noindent \textbf{Case 2.2:} Now consider $ \zeta_\mathcal{T}(\mathbb{S}_2) = v \geq 4 $, with the neighborhood of $ \mathbb{S}_2 $ defined as $ N(\mathbb{S}_2) = \{\mathbb{S}_1, \mathbb{S}_3, t_1, \ldots, t_{v-2}\} $. When we modify the tree to $ \mathcal{T}' = \mathcal{T} - \{\mathbb{S}_1\} $, the value of $ \varepsilon(\mathcal{T}) $ decreases, leading to $ \varepsilon(\mathcal{T}') = \varepsilon(\mathcal{T}) - 1 $. Thus, we obtain:
		
		\[
		\begin{aligned}
			\mathcal{M}_{2}(\mathcal{T}) &= \mathcal{M}_{2}(\mathcal{T}') + 2v, \\
			&\leq n^2 - 3n + 3 + \varepsilon(\mathcal{T}) - 2n + 3 + 2v, \\
			&= g_{\max}(n, \varepsilon(\mathcal{T})) - 2n + 2v + 3, \\
			&< g_{\max}(n, \varepsilon(\mathcal{T})).
		\end{aligned}
		\]
		
		Therefore, we conclude that the proof is complete.

	\end{proof}
	\section{Conclusion}
	In this paper, we investigated the influence of metric dimension on three key topological indices--the first Zagreb index, the second Zagreb index, and the $ABC$ index--within the class of trees. We established sharp upper and lower bounds for the Zagreb indices and derived an upper bound for the $ABC$ index, all in terms of the order and metric dimension of the trees. Our results revealed that path graphs attain the minimum Zagreb indices, while star graphs maximize both the Zagreb and $ABC$ indices, highlighting their extremal nature with respect to these measures.
	
	These findings provide deeper insights into how structural parameters, particularly metric dimension, affect the behavior of degree-based indices in tree graphs. Such insights have meaningful implications in chemical graph theory, where these indices serve as molecular descriptors in QSPR and QSAR models. The study contributes to the theoretical understanding of topological indices and lays the groundwork for future research into more complex graph families and the role of other structural parameters in influencing topological descriptors.

	\subsection*{Declarations}\noindent\textbf{Author Contribution Statement} All authors contributed equally to the paper.\\
	
	\noindent\textbf{Declaration of competing interest} The authors have no conflict of interest to disclose.\\
	
	\noindent\textbf{Data availability statements} All the data used to find the results is included in the manuscript.\\
	
	\vspace{0.5cm}
	\noindent\textbf{Acknowledgment} This research was supported by the Ministry of Higher Education (MOHE) through the Fundamental Research Grant Scheme (FRGS/1/2022/STG06/UMT/03/4).\\
	
	\noindent\textbf{Ethical statement} This article contains no studies with humans or animals.

\bibliographystyle{plain}
\bibliography{manuscript}

\end{document}